\documentclass[12pt]{amsart}
\usepackage{amssymb}

\newtheorem{theorem}{Theorem}

\newtheorem{lemma}{Lemma}
\newtheorem{cor}{Corollary}
\newtheorem{prob}{Problem}

\begin{document}

\title[One-sided power sum and cosine inequalities]
{One-sided power sum and cosine inequalities}
\author{Frits Beukers and Rob Tijdeman}
\address{Mathematical Institute, University Utrecht, P.O.Box 80010, 3508 TA Utrecht}
\email{f.beukers@uu.nl}
\address{Mathematical Institute, Leiden University, P.O.Box 9512, 2300 RA Leiden, The Netherlands}
\email {tijdeman@math.leidenuniv.nl}

\begin{abstract} In this note we prove results of the following types. Let be given distinct complex numbers $z_j$ satisfying the conditions
$|z_j| = 1, z_j \not= 1$ for $j=1, \dots, n$ and
for every $z_j$ there exists an $ i$ such that $z_i = \overline{z_j}. $
Then $$\inf_{k } \sum_{j=1}^n z_j^k \leq - 1. $$
If, moreover, none of the ratios $z_i/z_j$ with $i\ne j$
is a root of unity, then
$$\inf_{k } \sum_{j=1}^n z_j^k \leq - \frac {1} {\pi^4} \log n. $$
The constant $-1$ in the former result is the best possible.
The above results are special cases of upper bounds for $\inf_{k } \sum_{j=1}^n b_jz_j^k$ obtained in this paper.
\end{abstract}

%\dedicatory{}

%\thanks{}


\parindent=0pt

\maketitle

\section{Introduction}

Our colleague Marc N. Spijker asked the following question in view of an application in numerical analysis \cite{sp}:

\begin{prob}\label{prob1} Is it true that for given real numbers $b_j \geq 1$ and distinct complex numbers $z_j$ satisfying the conditions
$|z_j| = 1, z_j \not= 1$ for $j=1, \dots, n$ and
$$  for \ every \  z_j  \ there \ exists \ an \  i  \ such \ that \  z_i = \overline{z_j}, b_i=b_j $$
\noindent we have $\liminf_{k \to \infty} \sum_{j=1}^n b_j z_j^k \leq -1$?
\end{prob}

Note that by the conjugacy conditions on $b_i,z_i$ the sum $\sum_{j=1}^n b_jz_j^k$ is real for all $k$.
\ifx
It follows from Dirichlet's Simultaneous Approximation Theorem (\cite{hw}, Theorem 201) that under the above conditions   $\lim\sup_{k \to \infty} \sum_{j=1}^n b_j z_j^k = \sum_{j=1}^n b_j.$
Furthermore, Kronecker's Simultaneous Approximation Theorem (\cite{hw}, Theorem 443) implies that if $\pi$ and the arguments of the $z_j$'s are maximally linearly independent over the rationals, then
$$\liminf_{k \to \infty} \sum_{j=1}^n b_j z_j^k = - \sum_{j=1}^n b_j.$$
Here maximally independent means that from every pair of conjugates $z_i, z_j$ one is chosen.
%Tur\'an's one-sided inequalities (\cite{tu} Ch. 12)
%give an M depending on $z_1, \dots, z_n$ and a function $f(r,M)<0$ explicitly such that
%$$ \min_{k=r, r+1, \dots, r+M-1} \sum_{j=1}^n b_j z_j^k  <f(r,M),$$ provided that the numbers $z_j$ are away from 1 by at least some prescribed amount. However, since $\lim_{r \to \infty} f(r,M) = 0$
%it does not answer Spijker's question where a negative upper bound independent of $r$ and of the
%distances from the numbers $z_j$ to 1 is required.
\fi
\vskip.1cm

In Section 2 we answer Spijker's question in a slightly generalized and sharpened form (see Theorem \ref{thm1} and Corollary \ref{cor1}). The solution of Problem 1 has an application to numerical analysis,
more particularly Linear multistep methods (LMMs). They form a well-known class of numerical step-by-step methods for solving initial-value problems for certain systems of ordinary differential equations.
In many applications of such methods it is essential that the LMM has specific stability properties. An important property of this kind is named {\it boundedness}
and has recently been studied by Hundsdorfer, Mozartova and Spijker \cite{hms}. In that paper the {\it stepsize-coefficient} $\gamma$ is a crucial parameter in the study of boundedness.
In \cite{sp} Spijker attempts to single out all LMMs with a positive stepsize-coefficient $\gamma$ for boundedness. By using Corollary \ref{cor1} below he is able to nicely narrow the class of such LMMs.
\vskip.1cm

As a fine point we can remark that the bound $-1$ in Spijker's problem is the optimal one. Namely, take
$z_j=\zeta^j$ where $\zeta=e^{2\pi i/(n+1)}$ and $b_j=1$ for all $j$. Then the exponential sum equals
$n$ if $k$ is divisible by $n+1$ and $-1$ if not.
\vskip.1cm

If, moreover, none of $z_j/z_i$ with $i \neq j$ is a root of unity, then
the upper bound in Problem \ref{prob1} can be improved to $-\log n / \pi^4$.
We deal with this question in Theorem \ref{thm4} and more particularly Corollary \ref{cor3}.
The obtained results can easily be transformed into estimates for  $\inf_{k \in \mathbb{Z}} \sum_{j=1}^m b_j \cos (2 \pi\alpha_jk)$ where $b_j,\alpha_j$ are real numbers and $\alpha_1, \ldots, \alpha_n$ are strictly between $0$ and $1/2$.
Theorem \ref{continuous} states that this infimum is equal to $\inf_{t \in \mathbb{R}} \sum_{j=1}^m b_j \cos (2 \pi \alpha_jt)$, provided that the $\mathbb{Q}$-span of $\alpha_1, \ldots, \alpha_n$ does not contain 1.

\section{The general case}
We provide an answer to Problem \ref{prob1}.

\begin{theorem} \label{thm1}
Let $n$ be a positive integer. Let $b_1, \dots, b_n$ be nonzero complex numbers such that $b_{n+1-i}= \overline{b_i}$ for all $i=1,2, \dots, n.$
Let $z_1, \dots, z_n$ be distinct complex numbers with absolute value $1$, not equal to $1$, such that $z_{n+1-i}= \overline{z_i}$ for all $i=1,2, \dots, n.$ Then
$$ \liminf_{k \to \infty}  \sum_{j=1}^n b_j z^k_j \leq - \frac {\sum_{j=1}^n |b_j|^2} {\sum_{j=1}^n |b_j|}. $$
\end{theorem}
\noindent Note that $\sum_{j=1}^n b_jz_j^k $ is real because of the conjugacy conditions.

By applying the Cauchy-Schwarz inequality we immediately obtain the following consequence.

\begin{cor} \label{cor1}
Let $n,b_j,z_j$ be as in Theorem \ref{thm1}. Then
$$\liminf_{k \to \infty} \sum_{j=1}^n b_j z_j^k \leq - \frac 1n \sum_{j=1}^n |b_j|. $$
\end{cor}

Obviously this answers  Problem \ref{prob1}, since in that case $b_j\in\mathbb{R}_{\ge 1}$
for all $j$.
\vskip.1cm

In the special case when the $b_j$ are positive real numbers we can even drop
the distinctness condition on the $z_j$.

\begin{theorem}\label{thm2}
Let $n,b_j,z_j$ be as in Theorem \ref{thm1} with the additional condition $b_j\in\mathbb{R}_{>0}$
for all $j$ and let the distinctness condition on the $z_j$ be dropped. Then the conclusions
of Theorem \ref{thm1} and Corollary \ref{cor1} still hold.
\end{theorem}

Furthermore, since $\sum_{j=1}^n b_j z_j^k$ is almost periodic we can apply Dirichlet's Theorem on simultaneous diophantine approximation, and find that the liminf coincides with the infimum in the above
theorems and corollary.

\begin{proof} [Proof of Theorem \ref{thm1}.]
Put $s_{k} = b_1z_1^k +  \dots + b_n z_n^k$.
Put $c = - \lim\inf_{ k \to \infty} s_k$. Note that the $z_j^k$ equal $0$ on average for each $j$. Hence
the same holds for the values $s_k$ and thus we see that $c \geq 0$.
Let $\varepsilon > 0$. Choose $N$ so large that $s_k > -(c + \varepsilon)$ for all $k \geq N$.
For any positive integer $K$ consider the sum
$$ \Sigma_1 := \sum_{k=N}^{N+K-1} s_k.$$
Since none of the $z_i$ is 1, we have
$$ \Sigma_1 = \sum_{j=1}^n \sum_{k=N}^{N+K-1} b_j z_j^k = \sum_{j=1}^n b_j \frac {z_j^N - z_j^{N+K}} {1 - z_j}. $$
Thus there exists $C_1$, independent of $N$ and $K$, such that
$$
| \Sigma_1 | \leq \sum_{j=1}^n \frac {2 |b_j|} {|1 - z_j|} = C_1.
$$
Define $\Sigma_1^+$ to be the subsum of $\Sigma_1$ of all nonnegative $s_k$ and $\Sigma_1^-$ to be minus the subsum of $\Sigma_1$ of all negative $s_k$.
Let $P$ be the number of nonnegative $s_k$ for $k=N, \dots, N+K-1$. Then
$$ \Sigma_1^+ \leq \Sigma_1^- + C_1 \leq (K-P) (c+ \varepsilon) + C_1$$
and
$$ \Sigma_1^- \leq \Sigma_1^+ + C_1 \leq P \sum_{i=1}^n |b_i| + C_1.$$
Consider the sum $\Sigma_2 := \sum_{k=N}^{N+K-1} s_k^2$. Then, using $\overline{s_k}=s_k$,
\begin{eqnarray*}
\Sigma_2 &=& \sum_{k=N}^{N+K-1} |s_k|^2
= \sum_{i,j} \sum_{k=N}^{N+K-1} b_i \overline{b_j} z_i^k \overline{z_j}^k\\
&=& K \sum_{i=1}^n |b_i|^2 + \sum_{i \ne j} b_i \overline{b_j} \frac {(z_i/z_j)^N - (z_i/z_j)^{N+K}}
{ 1 - z_i/z_j}.
\end{eqnarray*}
Hence there exists $C_2$, independent of $N$ and $K$, such that
$$\left| \Sigma_2 - K \sum_{i=1}^n |b_i|^2 \right| \leq \sum_{i \ne j} |b_ib_j|
\frac {2} { | 1- z_i/z_j|} = C_2.$$
The terms $s_k^2$ in $\Sigma_2$ can be estimated above by $(\sum_{i=1}^n |b_i|) \Sigma_1^+$
when $s_k \geq 0$ and by $(c + \varepsilon) \Sigma_1^-$ when $s_k<0$. So we get the upper bound
$$ \Sigma_2 \leq (\sum _{i=1}^n |b_i|) \Sigma_1^+ + (c + \varepsilon) \Sigma_1^-.$$
Now use the upper bounds for $\Sigma_1^{\pm}$ we found above to get
$$ \Sigma_2 \leq (c+ \varepsilon) (K-P)(\sum_{i=1}^n |b_i|) + (c+ \varepsilon) P (\sum_{i=1}^n |b_i|) + C_3,$$
where $C_3 = C_1(c + \varepsilon + \sum_{i=1}^n  |b_i|).$
Combine this with the lower bound $\Sigma_2 \geq -C_2 + K \sum_{i=1}^n |b_i|^2$ to get
$$ -C_2 + K \sum_{i=1}^n |b_i|^2 \leq  (c + \varepsilon)K \sum_{i=1}^n |b_i| + C_3.$$
Dividing on both sides by $K$ and letting $K \to \infty$ yields
$$ \sum_{i=1}^n |b_i|^2 \leq (c+ \varepsilon) \sum_{i=1}^n |b_i|.$$
Since $\varepsilon$ can be chosen arbitrarily small, the assertion follows.

\end{proof}

\begin{proof} [Proof of Theorem \ref{thm2}.] Take the distinct elements from
$\{z_1,\ldots,z_n\}$ and write them as $w_1,\ldots,w_m$. Denote for any $r$ the sum of the $b_j$
over all $j$ such that $z_j=w_r$ by $B_r$. Then $s_k=\sum_{r=1}^m B_rw_r^k$ for every
$k$. We can apply Theorem \ref{thm1} to this sum to obtain
$$\liminf_{k\to\infty}s_k\le -\frac{\sum_{r=1}^m B_r^2}{\sum_{r=1}^m B_r}.$$
Since the $b_j$ are positive reals, we get that $\sum_{r=1}^m B_r=\sum_{j=1}^n b_j$
and $\sum_{r=1}^m B_r^2\ge \sum_{j=1}^n b_j^2$. Our theorem now follows.
\end{proof}

\section{The non-degenerate case}
In the next theorem we make an additional assumption on the $z_i$, which allows us to improve
on the upper bound in Theorem \ref{thm1} considerably.

\begin{theorem}\label{thm4} Let
$b_j\in\mathbb{C}$ and let $z_j\in\mathbb{C}$ be as in Theorem \ref{thm1}.
Assume in addition that $z_i\ne-1$ for all $i$ and that none of the ratios
$z_j/z_i$ with $i\ne j$ is a root of unity. Then
$$\inf_k \sum_{j=1}^nb_jz_j^k<- \frac {1} {\pi^4} (\min_j |b_j|) \log n.$$
\end{theorem}

If the $z_i$ satisfy the conditions of Theorem \ref{thm4} we say that we are in
the {\it non-degenerate case}. Notice in particular that $z_i/z_{n+1-i}=
z_i^2$ and hence none of the $z_i$ are roots of unity in the non-degenerate
case. For the proof of Theorem \ref{thm4} we use the following result.

\begin{lemma}\label{littlewood}
Let $b_1, \dots, b_n \in \mathbb{C}$. Let $q_1, \ldots, q_n$ be distinct integers. Set
$f(t) = \sum_{j=1}^n b_j e^{iq_jt}.$
Then
$$\int_{-\pi}^{\pi} |f(t)| dt \geq \frac {4} {\pi^3} (\min_j |b_j|) \log n.$$
\end{lemma}

\begin{proof} See Stegeman, \cite{st}.
\end{proof}

Lemma \ref{littlewood} is an refinement of a result independently obtained by McGehee, Pigno, Smith \cite{mps} and Konyagin \cite{ko} who thereby established a conjecture of Littlewood \cite{hl}.
Already Littlewood noticed that the constant in Lemma \ref{littlewood} cannot be better
than $4/ \pi^2$, cf. \cite{st} p. 51. Stegeman expects that the optimal constant in Lemma 
\ref{littlewood} is $4/ \pi^2$ indeed. As a fine point
we mention that the choice of $4/\pi^3$ by Stegeman is for esthetical reasons only,
the best possible value with his method happens to lie close to this value.
See also \cite{tr}.
\vskip.1cm
The following lemma connects Littlewood's
conjecture with minima of sums of exponentials.

\begin{lemma}\label{lemma1}
Let $b_1, \dots, b_n \in \mathbb{C}$. Let $q_1, \ldots, q_n$ be distinct nonzero integers.
Suppose that $f(t) = \sum_{j=1}^n b_j e^{iq_jt}$ is realvalued for all real $t$. Then
$$\min_{t\in\mathbb{R}}f(t) < -\frac {1} {\pi^4} (\min_j  |b_j|) \log n.$$
\end{lemma}

\begin{proof}
Denote the minimum of $f(t)$ by $-c$.
Define $f^+(t)=\max(f(t),0)$ and $f^-(t)=-\min(f(t),0)$. Then $f=f^+ - f^-$.
Since the exponents $q_j$ are nonzero, we have that
$\int_{-\pi}^{\pi}f(t)dt=0$, hence that $\int_{-\pi}^{\pi}f^+(t)dt=\int_{-\pi}^{\pi}f^-(t)dt$ and
$$\int_{-\pi}^{\pi}|f(t)|dt=\int_{-\pi}^{\pi}f^+(t)dt+
\int_{-\pi}^{\pi}f^-(t)dt=2\int_{-\pi}^{\pi}f^-(t)dt<4\pi c.$$
%Lemma \ref{littlewood} implies
%$$\int_{-\pi}^{\pi}|f(t)|dt> \frac {4} {\pi^3} \sum_{j=1}^n{|b_{n+1-j}|\over j}.$$
Now combine this upper bound with the lower bound from Lemma 1 to find the assertion of our lemma.
\end{proof}

\begin{proof}[Proof of Theorem \ref{thm4}.]
Consider the subgroup $G$ of $\mathbb{C} \setminus \{0\}$ generated by $z_1,\ldots,z_n$.
By the fundamental theorem of finitely generated abelian groups, $G$ is
isomorphic to $T\times\mathbb{Z}^d$ for some $d$ and some finite group $T$ consisting of roots of unity.
More concretely this means that there exist $w_1,\ldots,w_d\in G$ and $\mu\in T$
such that $w_1,\ldots,w_d$ are multiplicatively independent and every $z_j$
can be written in the form
$$z_j=\mu^{r_j}w_1^{a_{j1}}\cdots w_d^{a_{jd}},\quad r_j,a_{ji}\in\mathbb{Z},\quad
0\le r_j<|T|.$$
It follows from the condition in Theorem 1 that $a_{n+j-1,h} = - a_{j,h}$ for all $j = 1, \dots, n$ and $h = 1, \dots, d$.
Our exponential sum can be rewritten as
$$s_k:=\sum_{j=1}^nb_j\mu^{kr_j}w_1^{ka_{j1}}\cdots w_d^{ka_{jd}}.$$
By Kronecker's approximation theorem, the closure of the set of
points $(w_1^k,\ldots,w_d^k)$ for $k\in\mathbb{Z}_{\ge0}$ equals
the set $(S^1)^d$ consisting of points $(\omega_1,\ldots,\omega_d)$
with $|\omega_j|=1$ for $j=1,\ldots,d$. The same holds true if we
restrict ourselves to values of $k$ that are divisible by $|T|$. Hence
$$\inf s_k\le\min_{|\omega_1|=\cdots=|\omega_d|=1}
\sum_{j=1}^nb_j\omega_1^{a_{j1}}\cdots \omega_d^{a_{jd}}.$$
Because there are no roots of unity among the $z_j$, for every $j$ at least one coefficient
$a_{ji}$ is non-zero. Since the ratios $z_i/z_j$ are not a root
of unity for every $i\ne j$, the vectors $(a_{j1},\ldots,a_{jd})$ are pairwise
distinct. Hence we can choose $p_1,\ldots,p_d\in\mathbb{Z}$ such that the
numbers $q_j=a_{j1}p_1+\cdots+a_{jd}p_d,\ j=1,\ldots,n$ are distinct and nonzero.
Let us now restrict to the points with
$\omega_l=e^{ip_lt},t\in\mathbb{R},l=1,\ldots,d$. Then we get
$$\inf s_k\le \min_{t\in\mathbb{R}}\sum_{j=1}^n b_j e^{ iq_jt},$$
where the sum on the right-hand side is real for all $t$ in view of $b_{n+1-j} = \overline{b_j}, q_{n+1-j} = -q_j$ for $j=1, \dots, n$.
By Lemma \ref{lemma1} the right-hand side is bounded above by $- \frac{1} {\pi^4}( \min_j|b_j|)\log n.$
\end{proof}

In the special case $b_j=1$ for all $j$ we have the following corollary.

\begin{cor}\label{cor3} Let $z_1,\ldots,z_n$ be as in Theorem \ref{thm1}. Suppose in addition that
none of the ratios $z_i/z_j$ with $i\ne j$ is a root of unity. Then
$$\inf_{k\in\mathbb{Z}}\sum_{j=1}^nz_j^k<-{1\over\pi^4}\log n.$$
\end{cor}

\begin{proof}Note that we have not excluded the possibility that $z_i=-1$ for some $i$.
When $z_i\ne-1$ for all $i$ we are in the non-degenerate case and can apply Theorem \ref{thm4}
immediately. When $z_i=-1$ for some $i$ we can take $z_n=-1$. We now consider the
subsequence of sums for odd $k$. Put $k=2\kappa+1$. Note that
$$\inf_{k\in\mathbb{Z}}\sum_{j=1}^nz_j^k\le -1+\inf_{\kappa}\sum_{j=1}^{n-1}z_j^{2\kappa+1}.$$
Apply Theorem \ref{thm4} to the numbers $b_j=z_j$ for $j=1,\ldots,n-1$ and $z_j^2$ instead of
$z_j$ for $j=1,\ldots,n-1$. Then we find
$$\inf_{k\in\mathbb{Z}}\sum_{j=1}^nz_j^k\le -1-{1\over \pi^4}\log(n-1)<-{1\over\pi^4}\log n.$$
\end{proof}

\section{The continuous case}

The conditions on $b_j,z_j$ in Theorem \ref{thm1} can be seen as an invitation to write the
power sum as a cosine sum. We consider the easier case when $b_j\in\mathbb{R}$ for all $j$.
Then we have
$$\sum_{j=1}^nb_jz_j^k={\rm Re}\left(\sum_{j=1}^nb_jz_j^k\right)
=\sum_{j=1}^nb_j\cos(2\pi\alpha_jk),$$
where we have written $z_j=\exp(2\pi i\alpha_j)$ for all $j$. 
To make things simpler assume that $z_j\ne-1$ for all $j$. Then $n$ is even and
the arguments $\alpha_j$ come in pairs which are opposite modulo $\mathbb{Z}$. Letting
$m=n/2$ we rewrite our sum as
$$2\sum_{j=1}^m b_j\cos(2\pi\alpha_jk),$$
where we can assume that $0<\alpha_j<1/2$ for all $j$.

Corollary \ref{cor1} immediately implies the following.
\begin{cor}\label{cor2} Let $b_1,\ldots,b_m,\alpha_1,\ldots,\alpha_m$ be real numbers such that
the $\alpha_j$ are distinct and strictly between $0$ and $1/2$ for all $j$. Then
$$\inf_{k\in\mathbb{Z}}\sum_{j=1}^m b_j\cos(2\pi\alpha_jk)<-{1\over 2m}\sum_{j=1}^m|b_j|.$$
\end{cor}

\begin{proof} Apply Corollary \ref{cor1} with $n=2m$ and $b_j=b_{j-m}$ when $j>m$
and $z_j=\exp(2\pi i\alpha_j)$ when $j\le m$ and $z_j=\exp(-2\pi i\alpha_{j-m})$
when $j>m$.
\end{proof}

Similarly Theorem \ref{thm4} implies the following.
\begin{cor}Let $b_1,\ldots,b_m,\alpha_1,\ldots,\alpha_m$ be real numbers such that
the $\alpha_j$ are distinct and strictly between $0$ and $1/2$ for all $j$.
Suppose in addition that none of the differences $\alpha_i-\alpha_j$ with $i\ne j$
and none of the sums $\alpha_i+\alpha_j$ is rational. Then
$$\inf_{k\in\mathbb{Z}}\sum_{j=1}^m b_j\cos(2\pi \alpha_jk)<-{\log(2m)\over 2\pi^4}\min_j|b_j|.$$
\end{cor}

We introduce the notation
$$ c_S = -  \inf_{k \in \mathbb{Z} }\sum_{j=1}^m b_j \cos ( 2 \pi \alpha_jk), $$
$$ c_T = -  \inf_{t \in \mathbb{R} }\sum_{j=1}^m b_j \cos (2 \pi \alpha_jt). $$
In the notation $c_S, c_T$ we suppress the dependence on the $\alpha$'s and $b$'s. Of course, $ c_S \leq c_T$ for all numbers $\alpha_j$ and $b_j$.
\vskip.3cm

\noindent{\bf Problem 2.} {\it What are the corresponding upper bounds for $c_T$}?
\vskip.3cm

The following result  shows that $c_T=c_S$ under a general condition.

\begin{theorem}\label{continuous}
Let $b_1, \dots, b_m$ be real numbers and let $\alpha_1, \dots, \alpha_m$ be real numbers such that their $\mathbb{Q}$-span does not contain $1$. Then $c_S=c_T$.

\end{theorem}

In the proof we use the following consequence of Kronecker's theorem on simultaneous diophantine approximation.

\begin{lemma} \label{lemk}
Let $\alpha_1, \dots, \alpha_n$ be numbers such that their $\mathbb{Q}$-span does not contain $1$. Let $t_0 \in \mathbb{R}$. Given $\delta >0$ there exist integers $k, k_1, \dots, k_n$ such that
$|\alpha_jt_0 - \alpha_jk - k_j| < \delta$
for $j=1, \dots, n$.
\end{lemma}

\begin{proof}
 Let $\beta_1,\ldots,\beta_d$ be a basis of the $\mathbb{Q}$-vector
space spanned by the $\alpha_j$. Choose $\lambda_{ij}\in \mathbb{Q}$
such that
$$\alpha_j=\sum_{i=1}^d\lambda_{ij}\beta_i.$$
By a convenient choice of the $\beta_i$ we can see to it that $\lambda_{ij}\in \mathbb{Z}$
for all $i,j$.
Put $\Lambda=\max_j(\sum_i|\lambda_{ij}|)$. By Kronecker's theorem  (\cite{hw}, Theorem 442)
there exist integers $k,m_1,\ldots,m_d$ such that
$$|\beta_i t_0-\beta_i k - m_i|<\delta/\Lambda$$
for $i=1, \dots, d$.
Here we use the information that the $\mathbb{Q}$-span of the $\beta_i$'s does not contain $1$.
Put $k_j=\sum_{i=1}^d\lambda_{ij}m_i$. Then we get, for $j=1, \dots, n$,
\begin{eqnarray*}
|\alpha_jt_0-\alpha_jk-k_j|&\le&\sum_{i=1}^d|\lambda_{ij}|\cdot
|\beta_i t_0-\beta_i k-m_i|
<  \sum_{i=1}^d|\lambda_{ij}|\frac{\delta}{\Lambda} \le \delta.
\end{eqnarray*}
\end{proof}

\begin{proof} [Proof of Theorem \ref{continuous}]
It remains to prove that $c_T \leq c_S$.
Let $\varepsilon > 0$. Choose $t_0$ such that $$ 0 \leq c_T + \sum_{j=1}^m b_j \cos (2 \pi \alpha_jt_0)   <  \varepsilon .$$
We apply Lemma 3 with a $\delta$ which is so small that there exists an integer $k_0$ with
$$|\sum_{j=1}^m b_j\cos(2\pi\alpha_jt_0)-\sum_{j=1}^m b_j\cos(2\pi\alpha_jk_0)|<\varepsilon.$$
Hence
$$0\le c_T+\sum_{j=1}^m b_j\cos(2\pi\alpha_jk_0)<2\varepsilon.$$
We deduce that
\begin{eqnarray*}
c_T-c_S&=&c_T+\inf_{k\in \mathbb{Z}}\sum_{j=1}^m b_j\cos(2\pi\alpha_jk)\\
&\le&c_T+\sum_{j=1}^m b_j\cos(2\pi\alpha_jk_0)<2\varepsilon.
\end{eqnarray*}
Since $\varepsilon$ can be chosen arbitrarily close to zero, we conclude $c_S=c_T$.

%Let $\delta > 0$. By the above lemma there exist integers $k, k_1, \dots, k_d$ such that $|(t_0/D) \beta_r - k \beta_r - k_r| < \delta$ for $r=1, \dots, d$.
%Since $\sum_{j=1}^m b_j \cos (2 \pi t \sum_{r=1}^d \lambda_{jr} \beta_r )$ is a continuous function of $t$, we obtain, by choosing $\delta$ sufficiently small,
%$$ \sum_{j=1}^m b_j \cos (kD\alpha_j) = \sum_{j=1}^m b_j \cos (2 \pi \sum_{r=1}^d \lambda_{jr} kD\beta_r ) < $$
%$$\sum_{j=1}^m b_j \cos (2 \pi \sum_{r=1}^d \lambda_{jr} t_0\beta_r ) + \frac{\varepsilon}{3} = \sum_{j=1}^m b_j \cos (t_0\alpha_j) + \frac{\varepsilon}{3}- \sum_{j=1}^m b_j \cos (2 \pi  \lambda_{j0} t_0 ).$$
%Thus
%$$\sum_{j=1}^m b_j \cos ((kD) \alpha_j) < \inf_{t \in \mathbb{R}}  \sum_{j=1}^m b_j \cos (t\alpha_j) + \varepsilon,$$
%hence
%$$\inf_{k \in \mathbb{Z} }\sum_{j=1}^m b_j \cos (k \alpha_j) \leq  \inf_{t \in \mathbb{R}}  \sum_{j=1}^m b_j \cos (t\alpha_j)$$
%for every set $\{ \alpha_1, \dots, \alpha_m\}$ under consideration.
\end{proof}
\vskip.3cm

{\bf Acknowledgements} We thank Marc Spijker for making valuable comments on drafts of the paper and
providing the information about the application in numerical analysis. We are much indebted to the referee
for his careful reading of the manuscript and for saving us from two embarrassing mistakes.

\vskip.5cm
\end{document}